\documentclass[12pt]{amsart}
\usepackage{amssymb,amsmath,amsthm,latexsym}
\usepackage{mathrsfs}
\usepackage{a4wide}
\usepackage[usenames,dvipsnames]{color}
\usepackage{euscript}
\usepackage{graphicx}
\usepackage{mdwlist}
\usepackage{mathtools,dsfont,wasysym}
\usepackage[dvipsnames]{xcolor}

\DeclareFontFamily{U}{mathx}{\hyphenchar\font45}
\DeclareFontShape{U}{mathx}{m}{n}{
      <5> <6> <7> <8> <9> <10>
      <10.95> <12> <14.4> <17.28> <20.74> <24.88>
      mathx10
      }{}

\newcommand{\nn}[1]{{\vert\kern-0.25ex\vert\kern-0.25ex\vert #1 
    \vert\kern-0.25ex\vert\kern-0.25ex\vert}}
\newcommand{\abs}[1]{\vert #1\vert}

\newtheorem{theorem}{Theorem}[section]
\newtheorem*{theorema}{Theorem A}
\newtheorem*{theoremb}{Theorem B}
\newtheorem*{theoremc}{Theorem C}
\newtheorem{lemma}[theorem]{Lemma}
\newtheorem{corollary}[theorem]{Corollary}
\newtheorem{proposition}[theorem]{Proposition}
\theoremstyle{remark}
\newtheorem{remark}[theorem]{Remark}
\theoremstyle{definition}

\newtheorem{problem}[theorem]{Problem}

\numberwithin{equation}{section}
\makeatother

\newcommand{\vertiii}[1]{{\left\vert\kern-0.25ex\left\vert\kern-0.25ex\left\vert #1 
    \right\vert\kern-0.25ex\right\vert\kern-0.25ex\right\vert}}
\renewcommand{\leq}{\leqslant}
\renewcommand{\geq}{\geqslant}
\newcounter{smallromans}

\newenvironment{romanenumerate}
{\begin{list}{{\normalfont\textrm{(\roman{smallromans})}}}%
  {\usecounter{smallromans}\setlength{\itemindent}{0cm}%
   \setlength{\leftmargin}{5.5ex}\setlength{\labelwidth}{5.5ex}%
   \setlength{\topsep}{.5ex}\setlength{\partopsep}{.5ex}%
   \setlength{\itemsep}{0.1ex}}}%
{\end{list}}

\newcommand{\R}{\mathbb{R}}
\newcommand{\N}{\mathbb{N}}

\newcommand{\e}{\varepsilon}
\newcommand{\p}{\varphi}
\newcommand{\norm}{\left\Vert \cdot\right\Vert}

\newcounter{smallromansdash}

{\end{list}}

\newcounter{bigromans} 
  {\end{list}}

\begin{document}
\title[Symmetrically separated sequences]{Symmetrically separated sequences\\in the unit sphere of a Banach space}
\dedicatory{In memoriam: Joe Diestel (1943--2017)}
\subjclass[2000]{46B04, 46B20.}

\author[P.~H\'ajek]{Petr H\'ajek}
\address[P.~H\'ajek]{Mathematical Institute\\Czech Academy of Science\\\v Zitn\'a 25 \\115 67 Praha 1\\
Czech Republic and Department of Mathematics\\Faculty of Electrical Engineering\\
Czech Technical University in Prague\\ Jugosl\'avsk\'ych partyz\'an\r u 3, 166 27 Praha 6\\ Czech Republic}
\email{hajek@math.cas.cz}

\author[T.~Kania]{Tomasz Kania}
\address[T.~Kania]{Mathematics Institute, University of Warwick, Gibbet Hill Rd, Coventry, CV4 7AL, England and Mathematical Institute\\Czech Academy of Science\\\v Zitn\'a 25 \\115 67 Praha 1\\Czech Republic}
\email{tomasz.marcin.kania@gmail.com}

\author[T.~Russo]{Tommaso Russo}
\address[T.~Russo]{Dipartimento di matematica\\ Universit\`a degli Studi di Milano\\
via Saldini 50, 20133 Milano, Italy}
\email{tommaso.russo@unimi.it}

\thanks{Research of the first-named author was supported in part by GA\v CR 16-07378S, RVO: 67985840.
The second-named author acknowledges with thanks funding received from the European Research Council; ERC Grant Agreement No.~291497. 
Research of the third-named author was supported in part by the Universit\`a degli Studi di Milano (Italy) and in part by the Gruppo Nazionale per l'Analisi Matematica, la Probabilit\`a e le loro Applicazioni (GNAMPA) of the Istituto Nazionale di Alta Matematica (INdAM) of Italy.}

\keywords{Symmetrically separated unit vectors, Elton--Odell theorem, Kottman theorem, symmetric Kottman constant, unconditional basic sequence, cotype, spreading models}
\subjclass[2010]{Primary 46B04, 46B20; Secondary 46B15, 46B22}
\date{\today}

\begin{abstract}
We prove the symmetric version of Kottman's theorem, that is to say, we demonstrate that the unit sphere of an infinite-dimensional Banach space contains an~infinite subset $A$ with the property that $\|x\pm y\| > 1$ for distinct elements $x,y\in A$, thereby answering a question of J.~M.~F.~Castillo. In the case where $X$ contains an infinite-dimensional separable dual space or an unconditional basic sequence, the set $A$ may be chosen in a way that $\|x\pm y\| \geqslant 1+\varepsilon$ for some $\varepsilon > 0$ and distinct $x,y\in A$. Under additional structural properties of $X$, such as non-trivial cotype, we obtain quantitative estimates for the said $\varepsilon$. Certain renorming results are also presented.
\end{abstract}
\maketitle

\section{Introduction}
Kottman's theorem \cite{Kottman}, asserting that the unit sphere of an infinite-dimensional normed space contains a sequence of points whose mutual distances are strictly greater than one, sparked a new insight on the non-compactness of the unit ball in infinite dimensions. Elton and Odell \cite{E-O} employed methods of infinite Ramsey theory to improve Kottman's theorem significantly by showing that the unit sphere of an infinite-dimensional normed space contains a sequence $(x_n)_{n=1}^\infty$ such that $\|x_n - x_k\|\geqslant 1+\varepsilon$ ($k,n\in\N,k\neq n$) for some $\varepsilon > 0$. Even though the proof of Kottman's theorem has been greatly simplified over time (\cite[pp.~7--8]{diestel}), it was only recently when a new (still Ramsey-theoretic though) proof of the Elton--Odell theorem was obtained (\cite{FOSZ}). \smallskip

It is perhaps no surprise that the $\varepsilon$ appearing in the statement of the Elton--Odell theorem is intimately related to the geometry of the underlying space. Indeed, in the case of the space $\ell_p$ ($1\leqslant p<\infty$) it cannot be greater than the attained bound $2^{1/p}-1$ (see, \emph{e.g.}, \cite[p.~31]{atl}). Thus, studying geometric or structural properties of the space will often help in identifying possible lower bounds for separation constants of sequences in the unit sphere of the space. For example, Kryczka and Prus proved the quite remarkable result saying that in the unit sphere of a non-reflexive Banach space one may find a $\sqrt[5]{4}$-separated sequence (\cite{KrP}). Further quantitative estimates of the said lower bound expressed in terms of various moduli of convexity and related results may be found in \cite{CaGoPa, delpech, MaPa, prus,vN}.\smallskip

The main objective of the paper is to revisit and investigate the above-mentioned results in the setting of symmetric separation: let us say that a subset $A$ of a normed space is \emph{symmetrically $(\delta+)$-separated} (respectively, \emph{symmetrically $\delta$-separated}) when $\|x\pm y\| > \delta$ (respectively, $\|x\pm y\|\geqslant \delta$) for any distinct elements $x,y\in A$ ($\delta > 0$). J.~M.~F.~Castillo and P.~L.~Papini asked whether there is a symmetric version of the Elton--Odell theorem (\cite[Problem~1]{CaPa}), however according to Castillo (\cite{castillo}) prior to this research, it has not been known whether the unit sphere of an infinite-dimensional Banach space contains a~symmetrically (1+)-separated sequence. \smallskip

Castillo and Papini proved that the answer is affirmative for uniformly non-square spaces and for $\mathscr{L}_\infty$-spaces (consult \cite{CaPa} for more details). Also, although not stated explicitly, it follows from the proof of the main result of \cite{delpech} that the answer is affirmative for asymptotically uniformly convex spaces in which case the lower bound for the symmetric separation constant is expressed in terms of the so-called modulus of asymptotic uniform convexity. Certainly the unit spheres of both $\ell_1$ and $c_0$ contain symmetrically 2-separated sequences (in the former case plainly the standard vector basis is an example of such sequence, in the latter case one may take $x_n =- e_{n+1}+ \sum_{k=1}^n e_k$ ($n\in \mathbb N$), where $(e_n)_{n=1}^\infty$ denotes the unit vector basis of $c_0$). Consequently, if $X$ contains an isomorphic copy of either space, by the James distortion theorem, for every $\varepsilon \in (0,1)$ the unit sphere of $X$ contains a symmetrically ($1+\varepsilon$)-separated subset. (For more details see Lemma~\ref{obvious}.)  \smallskip

Our first main result is an extension of Kottman's theorem to symmetrically separated sequences. (The proofs of the results presented in the Introduction are postponed to subsequent sections, where the necessary terminology is also explained.)
\begin{theorema}[Symmetric version of Kottman's theorem]\label{symKottman} Let $X$ be an infinite-dimensional Banach space. Then the unit sphere of $X$ contains a symmetrically $(1+)$-separated sequence.
\end{theorema}
Subsequently, we identify classes of Banach spaces for which a symmetric version of the Elton--Odell theorem holds true. We prove that spaces containing boundedly complete basic sequences satisfy a symmetric version of the Elton--Odell theorem; this theorem will be the main ingredient of Theorem B in this paper.
\begin{theorem}\label{s-e.o. for bddly complete basic sequence} Let $X$ be a Banach space that contains a boundedly complete basic sequence. Then for some $\e>0$, the unit sphere of $X$ contains a symmetrically $(1+\e)$-separated sequence.
\end{theorem}
In a reflexive Banach space, every basic sequence is boundedly complete (\cite[Theorem~1]{james-bases}; see also \cite[Theorem~3.2.13]{ak}), hence we arrive at the following corollary.
\begin{corollary}Let $X$ be an infinite-dimensional reflexive Banach space. Then for some $\varepsilon > 0$ the unit sphere of $X$ contains a symmetrically $(1+\varepsilon)$-separated sequence. \end{corollary}
The above observation may be extended to more general spaces, as Johnson and Rosenthal proved that if $X$ is isomorphic to an infinite-dimensional subspace of a separable dual space, then it contains a boundedly complete basic sequence (\cite[Theorem~IV.1.(ii)]{JoRo}). Notably, spaces with the Radon--Nikodym property, or more generally, spaces with the so-called \emph{point-of-continuity property} (in short, PCP) contain separable dual spaces; consult \cite[Corollary~II.1]{GhMa} for the exact definition, the proof and the relation to the Radon--Nikodym property. We may record then the following corollary to Theorem~\ref{s-e.o. for bddly complete basic sequence}.

\begin{corollary}\label{PCP}Suppose that $X$ contains a subspace isomorphic to a subspace of a separable dual space. Then for some $\varepsilon > 0$ the unit sphere of $X$ contains a symmetrically $(1+\varepsilon)$-separated sequence. \smallskip

Consequently, the assertion holds true in the case where $X$ has the Radon--Nikodym property (or more generally, PCP).
\end{corollary}

For (a space isomorphic to) a Banach lattice $X$ we have the following three, not necessarily exclusive, possibilities: $X$ is reflexive, $X$ contains a subspace isomorphic to $c_0$ or $X$ contains a subspace isomorphic to $\ell_1$ (\emph{cf}.~\cite[Theorem~1.c.5]{LiTzaII}). By Theorem~\ref{s-e.o. for bddly complete basic sequence} and our considerations from the fourth paragraph we thus obtain the following result.
\begin{theoremb}\label{Banach lattices are good}Suppose that $X$ is a Banach space that contains an infinite-dimensional subspace isomorphic to a Banach lattice (for example, a space with an unconditional basis). Then for some $\varepsilon > 0$ the unit sphere of $X$ contains a symmetrically $(1+\varepsilon)$-separated sequence.\end{theoremb}
An infinite-dimensional Banach space is \emph{hereditarily indecomposable} when it does not contain any closed subspace $Y$ that may be decomposed as the direct sum $Y = W\oplus Z$ of two closed, infinite-dimensional subspaces $W,Z$. Gowers' Dichotomy Theorem (\cite{gowers}) asserts that an infinite-dimensional Banach space contains an infinite-dimensional subspace with an unconditional basis or a hereditarily indecomposable subspace. Certainly, hereditarily indecomposable spaces do not contain unconditional basic sequences.\smallskip

Even though the original example of a hereditarily indecomposable space was reflexive, there exist hereditarily indecomposable spaces without reflexive subspaces; the first example is due to Gowers \cite{gowersspace}. However, Gowers' space admits an equivalent uniformly Kadets--Klee norm (\cite[Corollary~10]{DiGiKu}), so it has PCP since it does not contain $\ell_1$ (\cite[Proposition~2]{DiGiKu}). Consequently, Corollary~\ref{PCP} applies to any renorming of Gowers' space. \smallskip

More recently, Argyros and Motakis (\cite{ArgMo}) constructed a $\mathscr{L}_\infty$-space $X_{{\rm AM}}$ without reflexive subspaces whose dual is isomorphic to $\ell_1$. In particular, $X_{{\rm AM}}$ is an Asplund space containing weakly Cauchy sequences that do not converge weakly, so the unit ball of $X_{{\rm AM}}$ is not completely metrisable in the relative weak topology. By \cite[Theorem~A]{EdWh}, $X_{{\rm AM}}$ fails PCP; the same reasoning applies to any closed subspace of $X_{{\rm AM}}$. (We are indebted to Pavlos Motakis for having explained this to us.) Nevertheless, $X_{{\rm AM}}$ being a $\mathscr{L}_\infty$-space, by a result of Castillo and Papini, contains a symmetrically $(1+\varepsilon)$-separated sequence in the unit sphere for some $\varepsilon > 0$.\smallskip

We then turn our attention to classes of spaces where a lower bound for the $\varepsilon$ appearing in the statement of Theorem B may be computed explicitly. To wit, we prove the following theorem.

\begin{theoremc}Let $X$ be an infinite-dimensional Banach space. Suppose that either
\begin{romanenumerate}
\item $X$ contains a normalised basic sequence satisfying a lower $q$-estimate for some $q<\infty$, or
\item $X$ has finite cotype $q$.
\end{romanenumerate}

Then for every $\e>0$ the unit sphere of $X$ contains a symmetrically $(2^{1/q}-\e)$-separated sequence.\end{theoremc}

\section{Preliminaries}
Normed spaces studied in this paper are either real or complex and the results are valid in both cases. Let $X$ be a normed space. We denote by $B_X$ and $S_X$, respectively the unit ball of $X$ and and the unit sphere of $X$.

\subsection{Two elementary lemmata} This short section is devoted to a very simple result that we require for the justification of the observation made in the introduction asserting that if $X$ contains an isomorphic copy of either $c_0$ or $\ell_1$, then for every $\varepsilon \in (0,1)$ the unit sphere of $X$ contains a symmetrically ($2-\varepsilon$)-separated subset. \smallskip

We shall make use of the following well-known inequality; its proof can be found for instance \cite[Lemma~2.2]{KaKo} or (in a slightly weaker formulation) in \cite[Lemma~6]{MaSwWe}. On the other hand, the present formulation or similar estimates can surely be found in older papers scattered throughout the literature. Let us just mention \cite[Lemma 3.1]{MaPa old}, where a very similar statement (actually, under slightly more general assumptions) can be found.
\begin{lemma}\label{lemma-ball-sphere} Let $X$ be a normed space. Suppose that $x,y$ are non-zero vectors in the unit ball of $X$. If $\|x-y\|\geq 1$, then
$$\left\Vert \frac{x}{\|x\|} - \frac{y}{\|y\|}\right\Vert \geq \|x-y\|.$$
\end{lemma}
It follows in particular that, if $\left\Vert \frac{x}{\|x\|} - \frac{y}{\|y\|}\right\Vert \leq 1$, then $\|x-y\|\leq 1$ too. Moreover, the lemma assures that in order to find a (symmetrically) (1+)-separated sequence of unit vectors it is in fact sufficient to find such a sequence in the unit ball, with no need to insist that all vectors are normalised (and the analogous assertion for $(1+\e)$-separation).

\begin{lemma}\label{obvious}
Let $(X,\norm_X)$ and $(Y,\norm_Y)$ be normed spaces and let $A\subseteq S_X$ be a (symmetrically) $(1+\e)$-separated set $(\e>0)$. Suppose that $T\colon X\to Y$ is an isomorphic embedding such that $\|T\|\cdot\|T^{-1}\|\leq 1+\delta$ $(\delta > 0)$. If $\delta\leq\frac{\e}{2+\e}$, then the set 
$$\tilde{A}:=\left\{\frac{Tx}{\|Tx\|_Y}\colon x\in A\right\}\subseteq S_Y$$
is (symmetrically) $(1+\e/2)$-separated.
\end{lemma}

\begin{proof}
Up to a scaling, we may assume without the loss of generality that for all $x\in X$ we have $\|x\|_X\leq\|Tx\|_Y\leq(1+\delta)\|x\|_X$. Consequently, $(1+\delta)^{-1}\cdot Tx\in B_Y$ for $x\in A$. Moreover, for distinct $x, y\in A$ we have
$$\left\Vert\frac{Tx}{1+\delta}-\frac{Ty}{1+\delta}\right\Vert_Y \geq \frac{1}{1+\delta}\cdot\|x-y\|_X\geq\frac{1+\e}{1+\delta}\geq 1.$$
Thus Lemma~\ref{lemma-ball-sphere} applied to the vectors $(1+\delta)^{-1}Tx$ and $(1+\delta)^{-1}Ty$ gives 
$$\left\Vert\frac{Tx}{\|Tx\|_Y}-\frac{Ty}{\|Ty\|_Y}\right\Vert_Y\geq
\left\Vert\frac{Tx}{1+\delta}-\frac{Ty}{1+\delta}\right\Vert_Y
\geq\frac{1+\e}{1+\delta}\geq 1+\e/2,$$
as $\delta\leq\frac{\e}{2+\e}$. The symmetric assertion is proved in the same way.
\end{proof}
Let us mention that in the above proof we made use of Lemma~\ref{lemma-ball-sphere} in order to obtain a~slightly better choice for $\delta$. If we only used the triangle inequality, we would have needed the condition $\delta\leq\frac{\e}{4}$.

\subsection{Boundedly complete basic sequences}
Here we collect a few remarks on boundedly complete basic sequences, that we shall use in the proof of Theorem~\ref{s-e.o. for bddly complete basic sequence}; for convenience of the reader, we start recalling the relevant definitions.\smallskip

A basic sequence $(e_j)_{j=1}^\infty$ in a Banach space $X$ is \emph{boundedly complete} if the series $\sum_{j=1}^\infty a^je_j$ converges in $X$ for every choice of the scalars $(a^j)_{j=1}^\infty$ such that
$$\sup_{k\in \mathbb N} \left\Vert \sum_{j=1}^k a^je_j \right\Vert <\infty.$$ 
It is very simple to verify that if $(e_j)_{j=1}^\infty$ is a boundedly complete basic sequence, then so is every block basic sequence of $(e_j)_{j=1}^\infty$.\smallskip

We shall require a small refinement of the classical Mazur technique of constructing basic sequences (see, \emph{e.g.}, \cite[Theorem~4.19]{FHHMZ} or \cite[Corollary~1.5.3]{ak}). In particular, we shall exploit along the way the following lemma due to Mazur (the formulation given here can be found, \emph{e.g.}, in \cite[Lemma~4.66]{HJ}). 
\begin{lemma}Let $E$ be a finite-dimensional subspace of a Banach space $X$ and $\e>0$. Then there exists a finite-codimensional subspace $F$ of $X$ such that for every $x\in E$ and $v\in F$
$$\|x\|\leq(1+\e)\|x+v\|.$$\end{lemma}

\begin{lemma}\label{lemma: basic sequence almost monotone} Let $X$ be an infinite-dimensional Banach space and let $(e_j)_{j=1}^\infty$ be a basic sequence in $X$. Suppose that $(\e_j)_{j=1}^\infty$ is a sequence of positive real numbers that converges to 0. Then there exists a block basic sequence $(x_j)_{j=1}^\infty$ of $(e_j)_{j=1}^\infty$, such that $$\|P_j\|\leq1+\e_j \quad (j\in \mathbb N),$$ where $P_j\colon\overline{\rm span}\{x_j\}_{j=1}^\infty\to\overline{\rm span}\{x_j\}_{j=1}^\infty$ denotes the j$^{th}$ canonical projection associated to the basic sequence $(x_j)_{j=1}^\infty$.\smallskip

In particular, if $(e_j)_{j=1}^\infty$ is boundedly complete, then so is $(x_j)_{j=1}^\infty$.
\end{lemma}

\begin{proof} Fix a sequence $\delta_j\searrow 0$ such that $\prod_{j=n}^\infty(1+\delta_j)\leq 1+\e_n$ for every $n$. We start choosing a unit vector $x_1\in {\rm span}\{e_j\}_{j=1}^\infty$. We then find a finite-codimensional subspace $F_1$ of $X$, obtained applying Mazur's lemma to ${\rm span}\{x_1\}$ and $\delta_1$. For $n_1$ sufficiently large, we have $x_1\in{\rm span}\{e_j\}_{j=1}^{n_1-1}$; since $F_1$ is finite-codimensional, we can choose a unit vector $x_2$ in $F_1\cap{\rm span}\{e_j\}_{j= n_1}^\infty$. By Mazur's lemma, for all scalars $\alpha^1,\alpha^2$ such $x_2$ satisfies
$$\|\alpha^1x_1\|\leq (1+\delta_1)\|\alpha^1x_1+\alpha^2x_2\|.$$

We proceed analogously by induction: assume that we have already found a finite block sequence $(x_j)_{j=1}^n$ of $(e_j)_{j=1}^\infty$ such that
$$\left\Vert \sum_{j=1}^k\alpha^jx_j \right\Vert \leq (1+\delta_k) \left\Vert \sum_{j=1}^{k+1}\alpha^jx_j \right\Vert$$
for every $k=1,\dots,n-1$ and scalars $\alpha^1,\dots,\alpha^n$. Let $F_n$ be a finite-codimensional subspace of $X$ as in the conclusion of Mazur's lemma, applied to ${\rm span}\{x_1,\dots,x_n\}$ and $\delta_n$. Moreover let $N\in\N$ be so large that $x_1,\dots,x_n\in{\rm span}\{e_j\}_{j=1}^{N-1}$. We can then choose a unit vector $x_{n+1}$ in $F_n\cap {\rm span}\{e_j\}_{j= N}^\infty$ and such a choice ensures us that
\begin{equation} \label{asympt.monot.}
\left\Vert \sum_{j=1}^n\alpha^jx_j \right\Vert \leq (1+\delta_n) \left\Vert \sum_{j=1}^{n+1}\alpha^jx_j \right\Vert
\end{equation}
for every choice of the scalars $\alpha^1,\dots,\alpha^{n+1}$. This concludes the inductive procedure.

From (\ref{asympt.monot.}) it is clear that for every $n,k\in\N$
$$\left\Vert \sum_{j=1}^n\alpha^jx_j \right\Vert \leq \prod_{j=n}^\infty(1+\delta_j) \left\Vert \sum_{j=1}^{n+k}\alpha^jx_j \right\Vert \leq (1+\e_n)\left\Vert \sum_{j=1}^{n+k}\alpha^jx_j \right\Vert,$$
hence $\|P_n\|\leq1+\e_n$. It is also clear from the construction that $(x_j)_{j=1}^\infty$ is a block basic sequence of $(e_j)_{j=1}^\infty$. Finally, the last assertion of the lemma follows from the comments preceding its proof.
\end{proof}

\section{Proof of Theorem A}
\begin{proof}[Proof of Theorem A] Let $X$ be an infinite-dimensional Banach space. We consider the following property of an infinite-dimensional subspace $\tilde{X}$ of $X$: $\tilde{X}$ has property (\Square) if there exist a unit vector $x\in S_{\tilde{X}}$ and an infinite-dimensional subspace $Y$ of $\tilde{X}$ such that $\|x+y\|>1$ for every unit vector $y\in S_Y$. In symbols,
$$\tilde{X} \text{ has (\Square) if:}\quad \exists x\in S_{\tilde{X}}, \exists Y\subseteq \tilde{X} \text{ infinite-dimensional subspace: } \forall y\in S_Y \, \|x+y\|>1.$$
Then we have the following dichotomy: either every infinite-dimensional subspace of $X$ has (\Square) or some infinite-dimensional subspace has ($\neg$\,\Square).

\smallskip
The proof of the result in the first alternative of the dichotomy is very simple: in fact, the assumption that $X$ has (\Square) yields a unit vector $x_1\in X$ and an infinite-dimensional subspace $X_1$ of $X$ such that $\|x_1+y\|>1$ for every $y\in S_{X_1}$. Since $X_1$ has (\Square) too, we can find a unit vector $x_2$ in $X_1$ and an infinite-dimensional subspace $X_2$ of $X_1$ such that $\|x_2+y\|>1$ for every $y\in S_{X_2}$. We proceed by induction in the obvious way and we find a sequence $(x_n)_{n=1}^\infty$ of unit vectors in $X$ and a decreasing sequence $(X_n)_{n=1}^\infty$ of infinite-dimensional subspaces of $X$ such that, for every $n\in\N$
\begin{romanenumerate}
\item $x_{n+1}\in X_n$ and
\item $\|x_n + y\| >1$ for every $y\in S_{X_n}$. 
\end{romanenumerate}
The sequence $(x_n)_{n=1}^\infty\subseteq S_X$ is then the desired symmetrically $(1+)$-separated sequence since for $1\leq k<n$ we have $\pm x_n\in X_{n-1} \subseteq X_k$; hence $\|x_k \pm x_n\|>1$.

\smallskip
In the second alternative, there exists an infinite-dimensional subspace $\tilde{X}$ of $X$ with property ($\neg$\,\Square); since we shall construct the desired sequence in the subspace $\tilde{X}$, we can assume without loss of generality that $\tilde{X}=X$. We first note that the assumption $X$ to admit property ($\neg$\,\Square) is equivalent to the formally stronger property
$$(\text{\XBox})\quad \forall x\in B_X, \forall Y\subseteq X \text{ infinite-dimensional subspace }\exists y\in S_Y\colon \|x+y\|\leq 1.$$
In fact, for $x\in S_X$, (\XBox) is exactly the negation of (\Square), while for $x=0$ it is trivially true. Given a non-zero $x\in B_X$ and an infinite-dimensional subspace $Y$ of $X$, ($\neg$\,\Square) provides us with a vector $y\in S_Y$ with $\left\Vert \frac{x}{\|x\|}+y\right\Vert\leq1$; consequently $\|x+y\|\leq 1$, by Lemma~\ref{lemma-ball-sphere}.

\smallskip
We finally prove the result under the additional assumption that $X$ has property (\XBox). Fix a decreasing sequence $(\delta_n)_{n=1}^\infty$ of positive reals with $\sum_{n=1}^\infty \delta_n\leq 1/4$, say $\delta_n=2^{-(n+2)}$. Also, choose any $z\in X$ with $\|z\|=3/4$ and find a norming functional $\psi\in S_{X^*}$ for $z$. \smallskip

We now construct by induction two sequences $(y_n)_{n=1}^\infty$ in $S_X$ and $(\p_n)_{n=1}^\infty$ in $S_{X^*}$ such that:
\begin{romanenumerate}
\item $\langle\p_n,y_n\rangle=1$ ($n\in\N$);
\item $y_1\in\ker\psi$ and $y_{n+1}\in\ker\psi\cap\bigcap_{i=1}^n\ker\p_i$ ($n\in\N$);
\item $\|z+y_1\|\leq1$ and $\|z- \delta_1y_1 - {\dots} - \delta_ny_n + y_{n+1}\|\leq 1$ ($n\in\N$).
\end{romanenumerate}
In fact, by (\XBox), there exists a unit vector $y_1\in\ker\psi$ such that $\|z+y_1\|\leq1$; we also find a~norming functional $\p_1$ for $y_1$. Assume that we have already found $y_1,\dots,y_n$ and $\p_1,\dots,\p_n$ for some $n\geq1$. Of course, the triangle inequality and our choice of $(\delta_n)_{n=1}^\infty$ imply
$$\|z- \delta_1y_1 - {\dots} - \delta_ny_n\|\leq 1,$$
thus (\XBox) ensures us of the existence of a unit vector $y_{n+1}$ in $\ker\psi\cap\bigcap_{i=1}^n\ker\p_i$ such that
$$\|z- \delta_1y_1 - {\dots} - \delta_ny_n + y_{n+1}\|\leq 1.$$
To complete the induction step it is then sufficient to take a norming functional $\p_{n+1}$ for $y_{n+1}$.\smallskip

We now define $x_1:=z+y_1$ and $x_{n+1}:=z- \delta_1y_1 - {\dots} - \delta_ny_n + y_{n+1}$ ($n\in\N$). Fix two natural numbers $k<n$. By the very construction, each $y_i$ lies in $\ker\psi$, so we have
$$\|x_n+x_k\|\geq\langle\psi,x_n+x_k\rangle=\langle\psi,2z\rangle=2\|z\|>1.$$
Moreover, $y_i\in\ker\p_k$ for every $i>k$, whence
$$\|x_k-x_n\|\geq\langle\p_k,x_k-x_n\rangle=\langle\p_k,y_k+\delta_ky_k+{\dots}+\delta_{n-1}y_{n-1}-y_n\rangle = \langle\p_k,(1+\delta_k)y_k\rangle=1+\delta_k>1.$$
Consequently, $(x_n)_{n=1}^\infty$ is a symmetrically $(1+)$-separated sequence and the vectors $x_n$ are contained in $B_X$, due to (iii). It thus follows from Lemma~\ref{lemma-ball-sphere} that the unit sphere of $X$ contains a symmetrically $(1+)$-separated sequence.
\end{proof}

\section{Symmetric $(1+\e)$-separation}
\subsection{Proof of Theorem~\ref{s-e.o. for bddly complete basic sequence}.}
We are ready to enter the proof of the theorem and we start introducing a bit of terminology. Given a basic sequence $(e_j)_{j=1}^\infty$, by a \emph{block} we mean a vector in ${\rm span}\{e_j\}_{j=1}^\infty$, and we also say that a block is a \emph{finitely supported} vector. A \emph{unit block} is of course a block which is also a norm one vector. Two blocks $b_1,b_2$ are \emph{consecutive} if $b_1\in {\rm span}\{e_j\}_{j=1}^N$ and $b_2\in {\rm span}\{e_j\}_{j=N+1}^\infty$; in this case we write $b_1 < b_2$. We also write $N<b$, where $N\in\N$, if $b\in{\rm span}\{e_j\}_{j=N+1}^\infty$, namely `the support of $b$ begins after $N$' (and analogously for $N\leq b$, $b<N$ or $b\leq N$). An \emph{extension} of a finite set of blocks $b_1<b_2<\dots<b_n$ is the choice of a block $b$ with $b_n < b$.

\begin{proof}[Proof of Theorem~\ref{s-e.o. for bddly complete basic sequence}.]Fix a boundedly complete basic sequence $(e_j)_{j=1}^\infty$ in $X$ and a decreasing sequence $(\e_j)_{j=1}^\infty$ of numbers in the interval $(0,1)$ with $\sum_{j=1}^\infty\e_j<\infty$. According to Lemma~\ref{lemma: basic sequence almost monotone} we can assume that the canonical projections $(P_j)_{j=1}^\infty$ associated to $(e_j)_{j=1}^\infty$ satisfy $\|P_j\|\leq1+\e_j$. We are going to construct the desired symmetrically separated sequence as a block basic sequence of $(e_j)_{j=1}^\infty$, so we can safely assume without loss of generality that $X={\rm span}\{e_j\}_{j=1}^\infty$. In other words, our actual assumptions are that $X$ admits a boundedly complete Schauder basis $(e_j)_{j=1}^\infty$, whose associated canonical projections satisfy $\|P_j\|\leq1+\e_j$.

\smallskip
We now begin with the construction. Either every symmetrically $(1+\e_1)$-separated finite family of unit blocks $b_1<b_2<\dots<b_n$ admits a symmetrically $(1+\e_1)$-separated extension $b_1<b_2<\dots<b_n<b$, with $b$ a unit block, or there exists a symmetrically $(1+\e_1)$-separated finite family of unit blocks $b_1<b_2<\dots<b_n$ that admits no such extension. In the first case we start with a family with cardinality $1$ and we can easily produce by induction a symmetrically $(1+\e_1)$-separated sequence $b_1 < b_2 <\dots< b_n < \dots$ consisting of unit blocks. In this case the proof is complete. 
Alternatively, we have found a finite family of unit blocks $\mathcal{B}_1:=(b_i^{(1)})_{i=1}^{N_1}$ which is symmetrically $(1+\e_1)$-separated and admits no extension with the same property. In other words, the family $\mathcal{B}_1$ satisfies: 
$$b_1^{(1)}<b_2^{(1)}<\dots<b_{N_1}^{(1)}, \; \|b_i^{(1)}\|=1, \;\|b_i^{(1)} \pm b_j^{(1)}\| \geq 1+\e_1\quad (i, j\in\{1,\dots,N_1\}, i\neq j)$$
and for every unit block $b > b_{N_1}^{(1)}$ there are $i=1,\dots,N_1$ and $\sigma=\pm1$ with $\|\sigma b_i^{(1)}+b\| < 1+\e_1$.\smallskip 

We next repeat the same alternative, but now we are in search of symmetrically $(1+\e_2)$-separated families of unit blocks and we only look for blocks $b>b_{N_1}^{(1)}$. Hence, either every symmetrically $(1+\e_2)$-separated finite family of unit blocks $b_1<b_2<\dots<b_n$ with $b_{N_1}^{(1)}<b_1$ admits a symmetrically $(1+\e_2)$-separated extension $b_1<b_2<\dots<b_n<b$, with $b$ a unit block, or there exists a symmetrically $(1+\e_2)$-separated finite family of unit blocks $b_1<b_2<\dots<b_n$ that admits no such extension. In the first case, the proof is completed by the simple induction argument, while in the second one we have obtained a~family $\mathcal{B}_2:=(b_i^{(2)})_{i=1}^{N_2}$ such that 
$$b_{N_1}^{(1)}<b_1^{(2)}<b_2^{(2)}<\dots<b_{N_2}^{(2)}, \, \|b_i^{(2)}\|=1\, \|b_i^{(2)} \pm b_j^{(2)}\| \geq 1+\e_2\quad (i, j\in\{1,\dots,N_2\}, i\neq j)$$
and for every unit block $b > b_{N_2}^{(2)}$ there are $i=1,\dots,N_2$ and $\sigma=\pm1$ with $\|\sigma b_i^{(2)}+b\| < 1+\e_2$. \smallskip

We proceed by induction in the obvious way: if at some step, say step $n$, we fall in the first of the two alternatives, then we easily conclude the existence of a symmetrically $(1+\e_n)$-separated sequence of unit vectors. In this case the proof is concluded and, of course, we stop our construction. In the other case, we tenaciously proceed for every $n$ and we consequently find families $\mathcal{B}_n:=(b_i^{(n)})_{i=1}^{N_n}$ such that for every $n\in\N$:
\begin{romanenumerate}
\item $\|b_i^{(n)}\|=1$ ($i=1,\dots,N_n$);
\item $b_1^{(n)}<b_2^{(n)}<\dots<b_{N_n}^{(n)}<b_1^{(n+1)}$;
\item $\|b_i^{(n)} \pm b_j^{(n)}\| \geq 1+\e_n$ ($i, j\in\{1,\dots,N_n\}, i\neq j$);
\item for any unit block $b > b_{N_n}^{(n)}$ there are $i=1,\dots,N_n$ and $\sigma=\pm1$ with $\|\sigma b_i^{(n)}+b\| < 1+\e_n$.
\end{romanenumerate}

\smallskip
Our plan now is to show that the existence of such families $(\mathcal{B}_n)_{n=1}^\infty$ is in contradiction with the assumption that $(e_j)_{j=1}^\infty$ is a boundedly complete Schauder basis. This implies that at some step we actually fall in the first alternative, and in turn this is sufficient to conclude the proof. The basic idea we exploit to implement our plan is to use elements of $\mathcal{B}_{n+1}$ to witness the non-extendability of $\mathcal{B}_n$. We will also use the following obvious inequality\footnote{We write it explicitly here only because it will be used with some slightly complicated expressions and it would be unpleasant to derive it every time with those expressions.}: if $a,b$ are vectors in a normed space $X$ and $1-\e\leq\|b\|\leq1+\e$, then
\begin{equation} \label{stupid bound}
\|a+b\| \leq \left\Vert a+\frac{b}{\|b\|}\right\Vert+\e.
\end{equation}

\smallskip
Fix any natural number $k\geq2$ and choose arbitrarily one index $n_k(k)\in\{1,\dots,N_k\}$; by condition (iv) there exist an index $n_{k-1}(k)\in\{1,\dots,N_{k-1}\}$ and a sign $\sigma_{k-1}(k)=\pm 1$ such that 
$$\left\Vert \sigma_{k-1}(k)b_{n_{k-1}(k)}^{(k-1)} + b_{n_k(k)}^{(k)}\right\Vert<1+\e_{k-1}.$$
Moreover, we can find an index $n$ with $b_{n_{k-1}(k)}^{(k-1)} \leq n< b_{n_k(k)}^{(k)}$ and clearly such $n$ satisfies $n\geq k-1$. Hence
$$1=\left\Vert b_{n_{k-1}(k)}^{(k-1)}\right\Vert=
\left\Vert P_n\left(\sigma_{k-1}(k)b_{n_{k-1}(k)}^{(k-1)} + b_{n_k(k)}^{(k)}\right) \right\Vert \leq (1+\e_{k-1})\left\Vert \sigma_{k-1}(k)b_{n_{k-1}(k)}^{(k-1)} + b_{n_k(k)}^{(k)}\right\Vert.$$
Consequently,
$$1-\e_{k-1}\leq\left\Vert\sigma_{k-1}(k)b_{n_{k-1}(k)}^{(k-1)} + b_{n_k(k)}^{(k)}\right\Vert<1+\e_{k-1}.$$
The vector
$$b:=\frac{\sigma_{k-1}(k)b_{n_{k-1}(k)}^{(k-1)} + b_{n_k(k)}^{(k)}}{\left\Vert\sigma_{k-1}(k)b_{n_{k-1}(k)}^{(k-1)} + b_{n_k(k)}^{(k)}\right\Vert}$$
is, of course, a unit block with $b>b_{N_{k-2}}^{(k-2)}$ and we can now use it to witness the maximality of $\mathcal{B}_{k-2}$. By condition (iv), there must exist an index $n_{k-2}(k)\in\{1,\dots,N_{k-2}\}$ and a sign $\sigma_{k-2}(k)=\pm 1$ such that $$\left\Vert \sigma_{k-2}(k)b_{n_{k-2}(k)}^{(k-2)} + b\right\Vert<1+\e_{k-2}.$$ By the inequality (\ref{stupid bound}) it then follows
$$1-\e_{k-2}\leq\left\Vert\sigma_{k-2}(k)b_{n_{k-2}(k)}^{(k-2)}+\sigma_{k-1}(k) 
b_{n_{k-1}(k)}^{(k-1)} + b_{n_k(k)}^{(k)}\right\Vert<1+\e_{k-2}+\e_{k-1}.$$
(where the lower bound is obtained applying a suitable projection $P_n$, as we have already done above). We proceed by going backwards in a similar way: the normalisation of the vector $\sigma_{k-2}(k)b_{n_{k-2}(k)}^{(k-2)}+\sigma_{k-1}(k)b_{n_{k-1}(k)}^{(k-1)} + b_{n_k(k)}^{(k)}$ is a unit block, which we use to witness the maximality of $\mathcal{B}_{k-3}$, and so on. In particular, we have proved the existence of a string of indices and signs $\{n_1(k),\sigma_1(k),\dots,n_k(k),\sigma_k(k)\},$ where $\sigma_k(k)=+1$, such that
$$\left\Vert\sigma_1(k)b_{n_1(k)}^{(1)}+\dots+
\sigma_k(k)b_{n_k(k)}^{(k)}\right\Vert<1+\e_1+\dots+\e_{k-1}.$$
If we apply again a suitable projection $P_n$, we also deduce the following stronger assertion: for every $k\in\N$ there exists a string of indices and signs $I_k=\{n_i(k),\sigma_i(k)\}_{i=1}^k$, where $\sigma_i(k)=\pm 1$ and $n_i(k)\in\{1,\dots,N_i\}$ for $i=1,\dots,k$, such that for every $\ell\in\N$, $\ell\leq k$ we have

\begin{equation} \label{bdd sums of blocks}
\left\Vert \sum_{i=1}^\ell \sigma_{i}(k)b_{n_{i}(k)}^{(i)}
\right\Vert\leq (1+\e_1)\cdot\left(1+\sum_{j=1}^\infty\e_j\right)=:C<\infty.
\end{equation}

Of course, there are only finitely many possibilities for the first two items of the strings $(I_k)_{k=1}^\infty$, so by the pigeonhole principle we may find an index $n_1\in\{1,\dots,N_1\}$ and a~sign $\sigma_1=\pm 1$ such that infinitely many strings begin with the pattern $\{n_1,\sigma_1\}$. Analogously, there are $n_2,\sigma_2$ such that an infinite subset of those strings begins with the pattern $\{n_1,\sigma_1,n_2,\sigma_2\}$. Continuing recursively, we find an infinite string $\{n_i,\sigma_i\}_{i=1}^\infty$ such that every its initial substring $\{n_i,\sigma_i\}_{i=1}^\ell$ is the initial part of infinitely many $I_k$'s. In particular, equation (\ref{bdd sums of blocks}) then implies that for every $\ell\in\N$
$$\left\Vert \sum_{i=1}^\ell \sigma_{i}b_{n_{i}}^{(i)}\right\Vert \leq C.$$

Finally, if we set $b_i:=\sigma_{i}b_{n_{i}}^{(i)}$, the sequence $(b_i)_{i=1}^\infty$ is a block basic sequence of $(e_j)_{j=1}^\infty$, hence it is boundedly complete. Moreover, the last inequality now reads
$$\sup_\ell \left\Vert\sum_{i=1}^lb_i\right\Vert\leq C.$$
It follows that the series $\sum_{i=1}^\infty b_i$ converges in $X$, which is a blatant contradiction with the fact that the $b_i$'s are unit vectors.
\end{proof}

\subsection{Quantitative results}
In this part we discuss some further results, in which it is possible to provide explicit estimates on the symmetric separation constant. In order to shorten the statements, we shall use the following \emph{symmetric Kottman constant}, introduced in \cite{CaPa}:
$$K^s(X):=\sup\Big\{\sigma>0\colon \exists(x_n)_{n=1}^\infty\subset B_X \,  \forall n\neq k\; \|x_n\pm x_k\|\geq\sigma \Big\}.$$

Let us start, for the sake of completeness, restating a few results already present in the literature concerning this constant. The first claim is probably a well known folklore fact, but we were not able to find it explicitly stated in the literature: if a Banach space $X$ admits a spreading model isomorphic to $\ell_1$, then $K^s(X)=2$. We shall say more on this in subsection \ref{Problem MaPa}, where we  will in particular briefly recall the notion of a spreading model and give a proof of such result.\smallskip

As we have already hinted at in the introduction, Castillo and Papini \cite[Proposition~3.4]{CaPa} proved that if $X$ is a $\mathscr{L}_\infty$-space, then $K^s(X)=2$. Delpech \cite{delpech} proved that every asymptotically uniformly convex Banach space $X$ satisfies $K^s(X)\geq1+\overline{\delta}_X(1)$, where $\overline{\delta}_X$ is the modulus of asymptotic uniform convexity (as we already mentioned, the symmetry assertion is not contained in the statement, but follows immediately from inspection of the proof). Prus \cite[Corollary~5]{prus} proved, among other things, that if $X$ has cotype $q<\infty$, then $K(X)\geq2^{1/q}$; it is not apparent from the argument whether it should also follow that $K^s(X)\geq2^{1/q}$. Therefore, we offer an alternative proof for this fact; our argument is based on an idea from \cite{KaKo}.\smallskip

A normalised basic sequence $(x_n)_{n=1}^\infty$ \emph{satisfies a lower $q$-estimate} if there is a constant $c>0$ such that
$$
c\cdot\left(\sum_{i=n}^N\abs{a_n}^q\right)^{1/q} \leq \left\|\sum_{n=1}^N a_nx_n\right\|
$$
for every choice of scalars $(a_n)_{n=1}^N$ and every $N \in \N$.\smallskip

Let $X$ be a Banach space with a basis $(x_n)_{n=1}^\infty$. Denote $X_n:= \overline{{\rm span}}\{x_i\}_{i= n}^\infty$ ($n\in \mathbb N$). We say that an operator $T\colon X\rightarrow Y$ is \emph{bounded by a pair $(\gamma,\varrho)$}, where $0<\gamma \leq \varrho < \infty$, if $\|T\|\leq\varrho$ and $\|T\vert_{X_n}\|\geq \gamma$ for every $n\in\N$.
\begin{proposition}\label{lowerq}Let $X$ be a Banach space that contains a normalised basic sequence satisfying a lower $q$-estimate. Then $K^s(X)\geq2^{1/q}$.
\end{proposition}

\begin{proof}Let $(x_n)_{n=1}^\infty$ be a normalised basic sequence with a lower $q$-estimate. We are going to construct the separated sequence as a block sequence of the basic sequence, so we can assume without loss of generality that $X=\overline{{\rm span}}\{x_i\}_{i=1}^\infty$. Then the assignment $Tx_n:=e_n$ ($n\in \mathbb N$) defines an injective, bounded linear operator $T\colon X\rightarrow\ell_q$.\smallskip

Set $\varrho_n=\|T\vert_{X_n}\|$ ($n\in \mathbb N$). Clearly, $(\varrho_n)_{n=1}^\infty$ is a decreasing sequence with $\varrho_n \geq1$ for every $n\in\N$. Moreover, $T\vert_{X_k}$ ($k\in \mathbb N$) is bounded by the pair $(\inf_{n\geq 1}\varrho_n,\varrho_k)$ and, of course, $\varrho_k\rightarrow \inf_{n\geq 1}\varrho_n$ as $k\rightarrow\infty$. In other words, up to replacing $X$ with $X_k$, for $k$ sufficiently large, we can (and do) assume that $T\colon X\rightarrow\ell_q$ is bounded by a pair $(\gamma,\varrho)$ with $\frac{\gamma}{\varrho}$ as close to $1$ as we wish (of course with $\frac{\gamma}{\varrho}<1$).\smallskip

Armed with this further information, we may now conclude the proof: let $\tilde{\gamma} < \gamma$ be such that $\frac{\tilde{\gamma}}{\varrho}$ is still as close to $1$ as we wish. Since $\|T\| >\tilde{\gamma}$, we can find a unit vector $y_1$ in ${\rm span}\{x_i\}_{i=1}^\infty$ 
such that $\|Ty_1\| > \tilde{\gamma}$. Assume now that we have already found unit vectors $y_1,\dots,y_n$ in ${\rm span}\{x_i\}_{i=1}^\infty$ such that $\|Ty_k\| > \tilde{\gamma}$ and the $Ty_k$ have mutually disjoint supports. Then there is $N$ such that $y_1,\dots,y_n \in {\rm span}\{x_i\}_{i=1}^N$ and the fact that $\|T\vert_{X_{N+1}}\| > \tilde{\gamma}$ allows us to find a unit vector $y_{n+1} \in {\rm span}\{x_i\}_{i=N+1}^\infty$ such that $\|Ty_{n+1}\| > \tilde{\gamma}$.\smallskip

Consequently, we have found a sequence $(y_n)_{n=1}^\infty$ in $S_X$ such that $\|Ty_n\| > \tilde{\gamma}$ and the supports of $Ty_n$ are finite and mutually disjoint. Hence for $n\neq k$ we have
$$\varrho \cdot \|y_n \pm y_k\| \geq \| Ty_n \pm Ty_k\| = \left( \|Ty_n\|^q + \|Ty_k\|^q \right)^{1/q} \geq \tilde{\gamma} \cdot 2^{1/q}.
$$
So 
$$K^s(X) \geq \frac{\tilde{\gamma}}{\varrho}\cdot 2^{1/q}
$$
and, since $\frac{\tilde{\gamma}}{\varrho}$ could be chosen to be as close to $1$ as we wish, the proof is complete.
\end{proof}

Recall that for a Banach space $X$ one sets 
$$q_X:=\inf \big\{q\in [2,\infty]: X\text{ has cotype }q\big\}.$$

\begin{corollary}Let $X$ be an infinite-dimensional Banach space. Then $K^s(X)\geq2^{1/q_X}$.
\end{corollary}

\begin{proof} If $q_X=\infty$, then the assertion follows immediately from the Riesz lemma, so we assume $q_X<\infty$. If $X$ is a Schur space, then by Rosental's $\ell_1$-theorem $X$ contains a copy of $\ell_1$ and the James' non-distortion theorem even implies $K^s(X)=2$. In the other case, there is a weakly null normalised basic sequence in $X$; it is known (see, \emph{e.g.}, \cite[Definition~3.54 and Proposition~4.36]{HJ}) that for every $r>q_X$ such a sequence admits a subsequence with a~lower $r$-estimate, so the result follows from the previous proposition.
\end{proof}

\section{Concluding remarks}
\subsection{Results under renorming}
In this short part, we observe that the problem of finding symmetrically $(1+\e)$-separated sequences of unit vectors is much easier if we allow renormings of the spaces under investigation. 

\begin{proposition}\label{renorm with 2-separation}Let $(X,\norm)$ be an infinite-dimensional Banach space. Then $X$ admits an equivalent norm $\vertiii{\cdot}$ such that $S_{(X,\vertiii{\cdot})}$ contains a symmetrically $2$-separated sequence.
\end{proposition}
This phenomenon was already observed by Kottman (\cite[Theorem~7]{Kottman}), who showed that every infinite-dimensional Banach space admits a renorming such that the new unit sphere contains a $2$-separated sequence. An inspection of his argument shows that actually the resulting sequence is symmetrically $2$-separated. We note in passing that van Dulst and Pach (\cite{vDP}) proved a stronger renorming result; however, we shall not require it here. For convenience of the reader, we  present a simple proof of Proposition~\ref{renorm with 2-separation}. 

\begin{proof}By the main result in \cite{Day}, $X$ contains an Auerbach system $\{x_i,f_i\}_{i=1}^\infty$. Set
$$
\nu(x) := \sup_{i\neq k \in\N}  (\left|\langle f_i,x\rangle\right|+\left|\langle f_k,x\rangle\right|)
$$
and let us define $$\vertiii{x} =\max\{\|x\|, \nu(x)\}\quad (x\in X).$$ Then $\vertiii{\cdot}$ is an equivalent norm on $X$ as $\|x\|\leq \vertiii{x} \leq 2\|x\|$ ($x\in X$). From the biorthogonality we deduce that $\nu(x_i)=1$, so $\vertiii{x_i} = 1$ ($i\in \N$). Moreover, $$\vertiii{x_i \pm x_j}\geq \nu (x_i \pm x_j)=2 \quad (i,j\in \N, i\neq j).$$ Hence, $(x_i)_{i=1}^\infty$ is a symmetrically 2-separated sequence in the unit sphere of $(X,\vertiii{\cdot})$.
\end{proof}

A modification of the above renorming yields a new norm $\vertiii{\cdot}$ that approximates $\norm$ and such that the unit sphere of $\vertiii{\cdot}$ contains a symmetrically $(1+\e)$-separated sequence. This shows how simple the symmetric version of the Elton--Odell Theorem would be if we were allowed to consider arbitrarily small perturbations of the original norm.
\begin{proposition}Let $(X,\norm)$ be an infinite-dimensional Banach space. Then, for every $\e>0$, $X$ admits an equivalent norm $\vertiii{\cdot}$ such that $\norm \leq \vertiii{\cdot} \leq (1+\e)\norm$ and $S_{(X,\vertiii{\cdot})}$ contains an infinite symmetrically $(1+\delta)$-separated subset, for some $\delta>0$.
\end{proposition}
In other words, for every infinite-dimensional Banach space, the set of all equivalent norms for which the symmetric version of the Elton--Odell theorem is true is dense in the set of all equivalent norms.
\begin{proof}The very basic idea is that in the definition of $\nu$ we replace the sum of the two terms by an approximation of their maximum. We may assume clearly that $\e \in (0,1)$ (and in this case we could actually choose $\delta=\e$); we then find a norm $\Phi$ on $\R ^2$ such that: 
\begin{romanenumerate}
\item $\norm _\infty \leq \Phi \leq (1+\e)\cdot\norm_\infty$;
\item $\Phi((1,0))=\Phi((0,1))=1$;
\item $\Phi((1,1))=1+\e$.
\end{romanenumerate}
For example, one can choose
$$
\Phi((\alpha,\beta)):= \max\left\{ \left\Vert (\alpha,\beta) \right\Vert_\infty
,(1+\e)\cdot {\left| \alpha + \beta \right|}{/2} \right\}.
$$
We also fix an Auerbach system $\{x_i,f_i\}_{i\in \N}$. Then we set
$$\nu(x) := \sup_{i\neq k \in \N}  \Phi( \left| \langle f_i , x \rangle \right| , \left| \langle f_k , x \rangle \right| )\quad(x\in X)$$
and, exactly as above, $$\vertiii{x} =\max\{\|x\|, \nu(x)\}\quad (x\in X).$$
Note that 
$$\nu(x)\leq (1+\e)\sup_{i\neq k \in \N} \max\{ \left| \langle f_i , x \rangle \right| , \left| \langle f_k , x \rangle \right| \} \leq (1+\e)\|x\|;$$
hence $\norm \leq \vertiii{\cdot} \leq (1+\e)\norm$.\smallskip

Finally, from the biorthogonality we deduce that $\nu(x_i)=1$ ($i\in \N$) and $\nu (x_i \pm x_j)=1+\e$ ($i,j\in \N$, $i\neq j$). Hence $\vertiii{x_i} = 1$ and $\vertiii{x_i \pm x_j}\geq 1+\e$ for $i\neq j$. Consequently, $(x_i)_{i=1}^\infty$ is a symmetrically $(1+\e)$-separated sequence in the unit sphere of $(X,\vertiii{\cdot})$.
\end{proof}

We conclude this part with the following remark, in sharp contrast with Proposition \ref{renorm with 2-separation}. It probably belongs to obvious mathematical folklore, but it fits so well here, that we could not resist the temptation of including it.

\begin{remark}\label{strictly convex and separation} Every separable Banach space admits a strictly convex renorming (even a~locally uniformly rotund one; see, \emph{e.g.}, \cite[Theorem 8.1]{FHHMZ}), so in particular the unit sphere under such renorming contains no $2$-separated sequences. Indeed, let $X$ be any Banach space and let $x,y\in S_X$ be linearly independent and $2$-separated vectors. Then $\frac{x-y}{2}$, the midpoint of the non-trivial segment joining $x$ and $-y$, is a point on the unit sphere of $X$; hence it is a witness that $X$ is not strictly convex. \smallskip 

The previous assertion is no longer true if the separability assumption is dropped. In fact, Partington (\cite[Theorem 1]{part}) showed that if $\Gamma$ is uncountable then every renorming of $\ell_\infty(\Gamma)$ contains an isometric copy of $\ell_\infty$. In particular, the unit sphere of every renorming of $\ell_\infty(\Gamma)$ contains a $2$-separated sequence.
\end{remark}

\subsection{On a problem by Maluta and Papini} \label{Problem MaPa}
In \cite[Theorem~2.6]{MaPa}, the authors show that for every Banach space $X$ one has $K(X)\leq2\cdot (1-\delta_X(1))$, where $\delta_X$ denotes the modulus of convexity of $X$. It follows in particular that every super-reflexive Banach space $X$ admits a renorming $\vertiii{\cdot}$ such that $K\left(X,\vertiii{\cdot}\right)<2$. They ask whether every space which fails to contain $c_0$ or $\ell_1$, or at least every reflexive space, admits a renorming with the Kottman constant smaller than $2$. To the best of our knowledge, and also according to the authors themselves, there seems to be no published solution to these questions. We thus take this occasion to mention explicitly the answers.

\smallskip
An example of a Banach space which does not contain isomorphic copies of either $c_0$ or $\ell_1$ and still has the Kottman constant equal to $2$ under every renorming is the Bourgain--Delbaen space $Y_{{\rm BD}}$ (\cite[Section~5]{BoDe}). $Y_{\rm BD}$ is the first example of a $\mathscr{L}_\infty$-space that is saturated by reflexive subspaces; in particular it contains no copy of $c_0$ or $\ell_1$. Still, every renorming of $Y_{{\rm BD}}$ has the Kottman constant (even $K^s$) equal to $2$ by \cite[Proposition~3.4]{CaPa} already quoted above. More generally, every predual of $\ell_1$ is another example of space for which the symmetric Kottman constant is equal to 2 under every renorming; we note that the space constructed by Argyros and Motakis is such an example, which does not contain $c_0$ either.

\smallskip
We next observe that if $X$ is any renorming of the Tsirelson space $T$ (actually, the space $T$ we consider is the one constructed by Figiel and Johnson \cite{FiJo}, see also \cite[Example~2.e.1]{LiTzaI}, and it is the dual to the original Tsirelson's space $T^*$ \cite{Tsirelson}), then $K^s(X)=2$. Our argument here will exploit the construction of spreading models, so let us briefly introduce this notion. The following important result is due to Brunel and Sucheston, \cite[Proposition~1]{BrSu}.

\begin{proposition} Let $(x_n)_{n=1}^\infty$ be a bounded sequence in a Banach space $X$. Then there exists a subsequence $(y_n)_{n=1}^\infty$ of $(x_n)_{n=1}^\infty$ such that for every $k\in\N$ and scalars $\alpha^1,\dots,\alpha^k$ the following limit exists:
$$\lim_{\begin{matrix} n_1<\dots<n_k \\ n_1\to\infty \end{matrix}} 
\left\Vert \sum_{i=1}^k \alpha^i y_{n_i} \right\Vert.$$
\end{proposition}

For a vector $(\alpha^i)_{i=1}^\infty\in c_{00}$, denote such a limit by $L\left((\alpha^i)_{i=1}^\infty\right)$; it is immediate to check that $L$ defines a seminorm on $c_{00}$ and that such a seminorm is actually a norm provided that the sequence $(y_n)_{n=1}^\infty$ is not convergent in $X$. In such a case, the completion of $c_{00}$ under the norm $L$ (which we will henceforth denote $\norm$) is called a \emph{spreading model} of $X$. Let us denote by $e_n=(\delta_n^j)_{j=1}^\infty$ ($n\in \mathbb N$) the $n\textsuperscript{th}$ vector of the canonical basis of $c_{00}$. The sequence $(e_n)_{n=1}^\infty$ will be called the \emph{fundamental sequence} of the spreading model. One its fundamental property, which is actually an obvious consequence of the definitions, is that the fundamental sequence is \emph{invariant under spreading}, \emph{i.e.}, for every choice of natural numbers $n_1<\dots<n_k$
$$\left\Vert \sum_{i=1}^k \alpha^i e_i \right\Vert=
\left\Vert \sum_{i=1}^k \alpha^i e_{n_i} \right\Vert.$$

We will be interested in the question when a Banach space $X$ admits a spreading model isomorphic to $\ell_1$. A first very simple consequence of the Rosenthal's $\ell_1$-theorem and the invariance under spreading is the following (see, \emph{e.g.}, \cite[Lemme~II.1.1]{BeaLa}): if $F$ is a spreading model of a Banach space $X$, then $F$ is isomorphic to $\ell_1$ if and only if the fundamental sequence of $F$ is equivalent to the canonical basis of $\ell_1$. The next characterization is due to Beauzamy, \cite[Theorem~II.2]{Beauzamy} (it may also be found in \cite[Theoreme~II.2.3]{BeaLa}).

\begin{proposition} Let $X$ be a Banach space. Then the following are equivalent:
\begin{romanenumerate}
\item $X$ admits a spreading model isomorphic to $\ell_1$;
\item there are $\delta>0$ and a bounded sequence $(x_n)_{n=1}^\infty$ in $X$ such that, for every $k\in\N$, $\e_1,\dots,\e_k=\pm1$ and $n_1<\dots<n_k$ one has
$$\frac{1}{k}\left\Vert \sum_{i=1}^k \e_i x_{n_i} \right\Vert \geq\delta;$$
\item for every $\eta>0$ there is a bounded sequence $(x_n)_{n=1}^\infty$ in $X$ such that, for every $k\in\N$, $\e_1,\dots,\e_k=\pm1$ and $n_1<\dots<n_k$ one has
$$1-\eta \leq \frac{1}{k}\left\Vert \sum_{i=1}^k \e_i x_{n_i} \right\Vert \leq 1+\eta.$$
\end{romanenumerate}
\end{proposition}

\begin{corollary} Suppose that a Banach space $X$ admits a spreading model isomorphic to $\ell_1$. Then for every renorming $\vertiii{\cdot}$ of $X$ one has $K^s(X,\vertiii{\cdot})=2$.
\end{corollary}

\begin{proof} From the equivalence between (i) and (ii) in the previous proposition, it is obvious that if $X$ admits a spreading model isomorphic to $\ell_1$, then the same occurs to $(X,\vertiii{\cdot})$. Hence, we only need to show that $K^s(X)=2$. Applying now (iii) of the same proposition yields, for every $\eta>0$, a sequence $(x_n)_{n=1}^\infty$ such that $$1-\eta \leq \|x_n\|\leq 1+\eta\quad\text{ and }\quad\frac{1}{2}\|x_n\pm x_k\|\geq 1-\eta\quad (n,k \in\N, n\neq k).$$ Consequently, the sequence $(\frac{x_n}{1+\eta})_{n=1}^\infty\subseteq B_X$ is symmetrically $2\cdot\frac{1-\eta}{1+\eta}$-separated. Lemma~\ref{lemma-ball-sphere} then yieds $K^s(X)\geq 2\cdot\frac{1-\eta}{1+\eta}$ and letting $\eta \to 0^+$ concludes the proof.
\end{proof}

It is known (see, \emph{e.g.}, \cite[Proposition~IV.2.F.2]{BeaLa}) that every spreading model of $T$ is isomorphic to $\ell_1$. We thus immediately infer the following corollary.

\begin{corollary} Let $T$ be the Tsirelson space. Then for every renorming $\vertiii{\cdot}$ of $T$ we have $K^s(T,\vertiii{\cdot})=2$.
\end{corollary}
In particular, we have an example of a reflexive Banach space every whose renorming has symmetric Kottman constant equal to $2$; this is the desired counterexample to the question in \cite{MaPa}.

\subsection{A few open problems} Of course the main unresolved question we should mention here is the validity of Theorem~B for every Banach space, which is already the first open problem in \cite{CaPa}.
\begin{problem}\label{sym E--O} Is the symmetric version of the Elton--Odell theorem valid for every Banach space? Namely, is it true that for every Banach space there are $\e>0$ and a symmetrically $(1+\e)$-separated sequence of unit vectors?
\end{problem}

From our results it follows that it would be sufficient to prove the result under the additional assumption that $X$ is hereditarily indecomposable or non-reflexive. In particular, a way to solve Problem~\ref{sym E--O} would be to find a symmetric version of the result by Kryczka and Prus. For this reason, we can also ask the following:
\begin{problem} Is there a constant $c>1$ such that for every non-reflexive Banach space $X$ one has $K^s(X)\geq c$?
\end{problem}

One further immediate deduction from Lemma \ref{lemma-ball-sphere} is that if $Z$ is an isometric quotient of a Banach space $X$, then $K^s(X)\geq K^s(Z)$; this is formally recorded \emph{e.g.} in \cite[Proposition 2.3]{KaKo}. In particular, every Banach space with an infinite-dimensional reflexive quotient contains a symmetrically $(1+\e)$-separated sequence of unit vectors, for some $\e>0$. If we combine this with the, already mentioned more than once, fact that every infinite-dimensional $\mathscr{L}_\infty$-space $X$ satisfies $K^s(X)=2$, we infer that a positive answer to the following problem would solve in the positive the main Problem \ref{sym E--O}.
\begin{problem}\label{dichotomy refl_quotient or L_inf} Does every infinite-dimensional Banach space either contain an infinite-dimensional $\mathscr{L}_\infty$-space or admit an infinite-dimensional reflexive quotient?
\end{problem}

Concerning spreading models, we have proved above that $K^s(X)=2$ whenever $X$ admits a spreading model isomorphic to $\ell_1$. We do not know whether an analogous result for $c_0$ holds true too.
\begin{problem} Suppose that a Banach space $X$ admits a spreading model isomorphic to $c_0$. Does it follow that $K^s(X)=2$?
\end{problem}
More in general, one may even ask whether the Kottman's constant of a Banach space is lower bounded by that of its spreading models, \emph{i.e.}, the following:
\begin{problem}\label{K lower bounded by spreading} Let $X$ be a Banach space and let $Z$ be a spreading model of $X$. Is it true that $K^s(X)\geq K^s(Z)$? Of course, the same question may be posed for $K(\cdot)$. 
\end{problem}

In Remark \ref{strictly convex and separation} we have mentioned the fact that whenever $\Gamma$ is an uncountable set, the space $\ell_\infty(\Gamma)$ has no strictly convex renorming; there actually exist examples of spaces with potentially smaller density character, for example $\ell_\infty/c_0$ (\cite{Bourgain}), that admit no strictly convex renorming. (One has to bear in mind that the space $\ell_\infty(\Gamma)$ has density character equal to $2^{|\Gamma|}$ as long as $\Gamma$ is infinite, however it may happen that in some models of set theory $2^{\aleph_0} = 2^{|\Gamma|}$ for all uncountable sets of cardinality less than the continuum.) This is related to a question of A.~Aviles (\cite[Question 7.7]{GaKu}) of whether there exists, without extra set-theoretic assumptions, a Banach space with density character $\aleph_1$ which has no strictly convex renorming. We may then ask the following related question.
\begin{problem} Does there exist in ZFC a Banach space $X$ with density character $\aleph_1$ such that the unit sphere of every renorming of $X$ contains a $2$-separated sequence?
\end{problem}

\subsubsection{Toroidally separated sequences}
Let $X$ be a complex normed space. We may naturally adjust the definition of symmetric separation to encompass complex number of modulus 1. Thus, let us call a set $A\subset X$ ($\delta+$)-\emph{toroidally separated} (respectively, $\delta$-\emph{toroidally separated}) ($\delta > 0$) when for all distinct $x,y\in A$ and complex numbers $\theta$ with $|\theta |=1$ we have $\|x-\theta y\|> \delta$ (respectively, $\|x-\theta y\|\geqslant \delta$). A quick inspection of Delpech's proof of the main theorem in \cite{delpech} reveals that the unit sphere of a complex asymptotically uniformly convex space contains a toroidally $(1+\e)$-separated sequence, for some $\e>0$. Similarly, Theorem~C has a natural counterpart in the complex case for toroidally separated sequences. It is then reasonable to ask whether the theorems of Kottman and Elton--Odell have such counterparts too.

\medskip{}
\textbf{Acknowledgments.} The authors wish to express their gratitude to the anonymous referee for the careful reading of the manuscript and for suggesting the questions presented here in Problems \ref{dichotomy refl_quotient or L_inf} and \ref{K lower bounded by spreading}. We are also grateful to J.~M.~F.~Castillo for many stimulating conversations concerning the topic of symmetric separation.

\end{document}